\documentclass[12pt]{article}

\usepackage{amsmath,amssymb,amsfonts,amsthm}

\newtheorem{theorem}{Theorem}[section]
\newtheorem{proposition}[theorem]{Proposition}
\newtheorem{lemma}[theorem]{Lemma}

\theoremstyle{definition}
\newtheorem{definition}[theorem]{\textbf{Definition}}

\newtheorem{remark}[theorem]{\textbf{Remark}}
\makeatletter\renewcommand\theenumi{\@roman\c@enumi}\makeatother

\newcommand{\R}{\mathbb{R}}

\newcommand{\Hy}{\mathbb{H}^3}

\newcommand{\LH}{\mathcal{L}}

\newcommand{\M}{\mathcal{M}}

\begin{document}

\date{ }
\title{Global smooth geodesic foliations \\
of the hyperbolic space}
\author{Yamile Godoy and Marcos Salvai~\thanks{%
Partially supported by \textsc{CONICET, FONCyT, SECyT~(UNC)}.} \\
{\small {FaMAF - CIEM}, Ciudad Universitaria, 5000 C\'{o}rdoba, Argentina}\\
ygodoy@famaf.unc.edu.ar, salvai@famaf.unc.edu.ar}
\maketitle

\begin{abstract}
We consider foliations of the whole three dimensional
hyperbolic space $\mathbb{H}^{3}$ by oriented geodesics. Let $\mathcal{L}$
be the space of all the oriented geodesics of $\mathbb{H}^{3}$, which is a
four dimensional manifold carrying two canonical pseudo-Riemannian metrics
of signature $\left( 2,2\right) $. We characterize, in terms of these
geometries of $\mathcal{L}$, the subsets $\mathcal{M}$ in $\mathcal{L}$ that
determine foliations of $\mathbb{H}^{3}$. We describe in a similar
way some distinguished types of geodesic foliations of $\mathbb{H}^{3}$,
regarding to which extent they are in some sense trivial in some directions:
On the one hand, foliations whose leaves do not lie in a totally geodesic
surface, not even at the infinitesimal level. On the other hand, those for
which the forward and backward Gauss maps $\varphi ^{\pm }:\mathcal{M}\rightarrow \mathbb{H}%
^{3}\left( \infty \right) $ are local diffeomorphisms. Besides, we prove
that for this kind of foliations, $\varphi ^{\pm }$ are global
diffeomorphisms onto their images.

The subject of this article is within the framework of foliations by
congruent submanifolds, and follows the spirit of the paper by Gluck and
Warner where they understand the infinite dimensional manifold of all the
great circle foliations of the three sphere.
\end{abstract}


\section{Introduction}

\subsection{Geodesic foliations}

A smooth geodesic foliation of a Riemannian manifold $N$ is given by a
smooth unit vector field $V$ on $N$ all of whose integral curves, the
leaves, are geodesics. Throughout the paper, smooth means of class $\text{C}^\infty$. 

The standard examples of geodesic foliations of $\mathbb{R}^3$ are given by
foliating the space by parallel planes which are in turn foliated by
parallel lines, with smoothly varying directions. One can construct other
examples by writing $\mathbb{R}^{3}$ smoothly as the disjoint union of the $z
$-axis and one-sheet hyperboloids of revolution around that axis, and
considering on each one, coherently, one of the two ways of ruling it (the
striction circles of the hyperboloids do not need to be at the same height).
Applying a linear isomorphism one obtains new examples.

Notice that the foliations of the hyperbolic three space by totally geodesic
surfaces, as well as the foliations of the hyperbolic plane by geodesics,
are by far not as rigid as in the Euclidean case \cite{Ferus-1973, Browne-1984}. So, 
the hyperbolic analogues of the standard examples of Euclidean geodesic foliations are
richer. Recently, Nuchi studied the fiberwise homogeneous geodesic foliations of the three dimensional space forms \cite{Nuchi}.

Global smooth geodesic foliations of the three dimensional Euclidean space
were characterized in \cite{Salvai-2009} in terms of the geometry of the
space of oriented lines. Now, we deal with the analogous problem in the
hyperbolic context. The general basic theory for the Euclidean case is still
useful, but some crucial definitions and arguments in the proofs must be
adapted to the hyperbolic setting.

Let $\mathbb{H}^{3}$ be the three dimensional hyperbolic space of constant
sectional curvature $-1$. Let $\mathcal{L}_{0}$ and $\mathcal{L}_{-}$ be the
spaces of oriented geodesics of $\mathbb{R}^{3}$ and $\mathbb{H}^{3}$,
respectively, which are manifolds of dimension four admitting canonical
neutral pseudo-Riemannian metrics:\ $\mathcal{L}_{0}$ admits one (associated
with the cross product) \cite{Guilfoyle-2005, Salvai-2005}, and $\mathcal{L}_{-}$ 
admits two of them, $g_{\times }$ and $g_{K}$, coming from the cross
product and the Killing form on Iso$\,\left( \mathbb{H}^{3}\right) $,
respectively \cite{Salvai-2007, Georgiou-2010}. See the precise definitions
below in the preliminaries. Distinguished geometries on spaces of oriented
geodesics are also studied in \cite{Alekseevsky-2011} and \cite{Anciaux-2012}.

While the geodesic foliations of $\mathbb{R}^{3}$ are described in terms of
the canonical neutral metric on $\mathcal{L}_{0}$, the characterization of
the geodesic foliations of $\mathbb{H}^{3}$ involves both $g_{\times }$ and 
$g_{K}$. This situation appears also in other problems in hyperbolic
geometry; for instance, A.\ Honda needed both canonical neutral metrics on 
$\mathcal{L}_{-}$ in the study of the isometric immersions of the hyperbolic
plane into $\mathbb{H}^{3}$ \cite{Honda-2012}.

We will have two types of distinguished foliations. We call a geodesic
foliation $\emph{nondegenerate}$ if the leaves do not lie in a totally
geodesic surface, not even at the infinitesimal level. More precisely, if
the only eigenvectors of $\nabla V$ are in $\mathbb{R}V$, where $V$ is the
unit vector field that determines the foliation.

We have also another notion, which turns out to be weaker: a 
\emph{semi-nondegenerate} foliation does not resemble, in any direction, a trivial
foliation whose leaves are all orthogonal to a fixed horosphere. In the
upper half space model of $\mathbb{H}^{3}$, these are foliations congruent
to the those with vertical geodesics, with both orientations. See Definition
\ref{nondeg}. Both concepts generalize the Euclidean notion of nondegeneracy
(see Corollary 4 in \cite{Salvai-2009}). For a higher dimensional (local) analogue, see the foliations 
of  $\R^{n}$ by pairwise skew $p$-planes in \cite{Tabachnikov-2013}.

We want to emphasize that the statements of the results are similar to those
of \cite{Salvai-2009}, but the technical meaning of the definitions involved, 
for instance, (almost) semidefinite submanifolds and (semi-)nondegenerate foliations are quite
different in the Euclidean and the hyperbolic cases.

\subsection{Foliations by congruent submanifolds}

The general setting of this article is the study of foliations of a smooth
manifold $N$ by congruent submanifolds: Suppose that the Lie group $G$ acts
on $N$, let $M$ be a closed submanifold of $N$, and let $\mathcal{C}$ be the
set of all submanifolds of $N$ congruent to $M$ via $G$. Let $H$ be the
subset of all points in $G$ that preserve $M$. Then $H$ is a closed Lie subgroup of 
$G$ since $M$ is closed in $N$ and we can identify $\mathcal{C}\cong G/H$.
Sometimes $\mathcal{C}$ admits distinguished $G$-invariant geometries, which
are useful to describe the foliations of $N$ by submanifolds congruent to $M$. 
More precisely, the problem is the following: 
\begin{center}
Describe geometrically which subsets $\mathcal{M}$ of $\mathcal{C}$
determine foliations of $N$.
\end{center}

The paradigm is the paper \cite{Gluck-1983}, where foliations of $S^{3}$ by
great circles are characterized in this way. See also \cite{Salvai-2002} (a
partial generalization of \cite{Gluck-1983}) and \cite{Salvai-2009}, with
the global foliations of $\mathbb{R}^{3}$ by oriented lines, which includes
a pseudo-Riemannian reformulation of the principal result of \cite%
{Gluck-1983}. 
M. Czarnecki and R. Langevin are currently working, in this
context, on the classification of codimension two totally geodesic
foliations of the complex hyperbolic space. 

\section{Preliminaries}
A smooth geodesic foliation of $\mathbb{H}^{3}$ is given by a smooth unit
vector field $V$ on $\mathbb{H}^{3}$ all of whose integral curves, the
leaves, are geodesics. The set $\mathcal{M}$ of all the leaves admits a
canonical differentiable structure. For the sake of completeness, we include
its existence as a proposition.

\begin{proposition}\label{est.dif}
The set $\mathcal{M}$ of all the leaves of a geodesic foliation of $\mathbb{%
H}^{3}$ admits a unique differentiable structure such that the canonical
projection $P:\mathbb{H}^{3}\rightarrow \mathcal{M}$ is a smooth submersion.
\end{proposition}

\begin{proof}
Let $V$ be the smooth unit vector field on $\mathbb{H}^{3}$ associated with
the geodesic foliation and consider the smooth distribution given by $%
\mathcal{D}=\mathbb{R}V$. By Theorem VIII in \cite{Palais-1957}, it suffices
to prove that $\mathcal{D}$ is regular, that is, if for each $p\in \mathbb{H}%
^{3}$ there is a cubical coordinate system $(U,(x_{1},x_{2},x_{3}))$
centered at $p$ such that $\left\{ \left. {(\partial /\partial x_{3}}%
)\right\vert _{q}\right\} $ is a basis of $\mathcal{D}_{q}$ for all $q\in U$
and each leaf of $\mathcal{D}$ intersects $U$ in at most one 1-dimensional
slice $\left( x_{1},x_{2}\right) =$ const.

Let us see that for each $p\in \mathbb{H}^{3}$ we have such a coordinate
system. Let $u_{1},u_{2}\in T_{p}\mathbb{H}^{3}$ such that $\left\{
u_{1},u_{2},V(p)\right\} $ is an orthonormal basis of $T_{p}\mathbb{H}^{3}$
and let $F:\mathbb{R}^{2}\rightarrow \mathbb{H}^{3}$ be the totally geodesic
submanifold given by $F\left( x,y\right) =$ Exp$_{p}(xu_{1}+yu_{2})$ (here $%
\text{Exp}$ is the geodesic exponential map). We consider the smooth map 
\begin{equation*}
\alpha :\mathbb{R}^{3}\rightarrow \mathbb{H}^{3},\hspace{0.5cm}\alpha
(x,y,t)=\vartheta _{t}(F\left( x,y\right) )\text{,}
\end{equation*}%
where $\vartheta _{t}$ is the flow of $V$. Since $d\alpha _{0}$ is an
isomorphism, there exist $\varepsilon >0$ and an open neighborhood $U\subset 
\mathbb{H}^{3}$ of $p$ such that ${\alpha }:(-\varepsilon ,\varepsilon
)^{3}\rightarrow U$ is a diffeomorphism. Hence, $\left( U,\alpha
^{-1}=(x_{1},x_{2},x_{3})\right) $ is a cubical coordinate system centered
at $p$ such that $\left. {(\partial /\partial x_{3}})\right\vert _{q}=V(q)$
for all $q\in U$. The 1-dimensional slices are clearly integral submanifolds
od $\mathcal{D}$ and no leaf of $\mathcal{D}$ intersects two different
slices of $U$, since geodesics in $\mathbb{H}^{3}$ transverse to a totally
geodesic surface intersect it at most at one point.
\end{proof}

\medskip

The space $\mathcal{L}$ of all complete oriented geodesics of $\mathbb{H}^{3}
$ (up to orientation preserving reparametrizations) admits a unique
differentiable structure such that the canonical projection $\Pi :T^{1}%
\mathbb{H}^{3}\rightarrow \mathcal{L}$ is a differentiable submersion (by 
\cite{Palais-1957}, as above, with the spray as the vector field giving the
foliation). We may think of $c\in \mathcal{L}$ as the equivalence class of
unit speed geodesics $\gamma :\mathbb{R}\rightarrow \mathbb{H}^{3}$ with
image $c$ such that $\{\dot{\gamma}(t)\}$ is a positive basis of $T_{\gamma
(t)}c$ for all $t$. If $\ell \in \mathcal{L}$, then by abuse of notation we
sometimes write $z\in \ell $, meaning that $z$ is in the underlying line.

Fixing a point $o\in \mathbb{H}^{3}$, let 
\begin{equation}
H:T(T_{o}^{1}\mathbb{H}^{3})\rightarrow \mathcal{L}  \label{H}
\end{equation}%
be the map defined as follows: Let $u\in T_{o}^{1}\mathbb{H}^{3}$ and $v\in
T_{o}\mathbb{H}^{3}$ with $u\bot v$, then $H(u,v)$ is the oriented geodesic
with initial point $\text{Exp}_{o}(v)$ and initial velocity the parallel
transport of $u$ along the geodesic $t\mapsto \text{Exp}_{o}(tv)$ at $t=1$. 
Proposition 4.14 of \cite{Beem-1996} asserts that $H$ is a diffeomorphism.

Let $\gamma $ be a complete unit speed geodesic of $\mathbb{H}^{3}$ and let $%
\mathcal{J}_{\gamma }$ be the space of all Jacobi vector fields along $%
\gamma $ which are orthogonal to $\dot{\gamma}$. There exists a well-defined
canonical isomorphism 
\begin{equation}
T_{\gamma }:\mathcal{J}_{\gamma }\rightarrow T_{[\gamma ]}\mathcal{\mathcal{L%
}}\text{,}\hspace{1cm}T_{\gamma }(J)=\left. {\frac{d}{dt}}\right\vert
_{0}[\gamma _{t}]\text{,}  \label{isoT}
\end{equation}%
where $\gamma _{t}$ is any variation of $\gamma $ by unit speed geodesics
associated with $J$ (see \cite{Salvai-2007}).

Given a tangent vector $X$ to a pseudo-Riemannian manifold, we denote $%
\left\Vert X\right\Vert =\left\langle X,X\right\rangle $, the square norm
associated with the metric $\left\langle .\,,.\right\rangle $, and $\left\vert
X\right\vert =\sqrt{\left\vert \left\langle X,X\right\rangle \right\vert }$.
Also, given $v\in T\mathbb{H}^{3}$, we denote by $\gamma _{v}$ the unique
geodesic in $\mathbb{H}^{3}$ with initial velocity $v$.

Now, we recall the definition of the two canonical pseudo-Riemannian metrics 
$g_{\times }$ and $g_{K}$ on $\mathcal{\mathcal{L}}$ given in \cite[Theorem 1%
]{Salvai-2007}. In terms of the isomorphism (\ref{isoT}) the square norms of
these metrics may be written as follows \cite[page 362]{Salvai-2007}: For $%
J\in \mathcal{J}_{\gamma }$, 
\begin{equation}
\begin{array}{lll}
\Vert T_{\gamma }(J)\Vert _{\times } & = & \langle \dot{\gamma}\times
J,J^{\prime }\rangle \text{,} \\ 
\Vert T_{\gamma }(J)\Vert _{K} & = & \left\vert J\right\vert
^{2}-\left\vert J^{\prime }\right\vert ^{2}\text{.}%
\end{array}
\label{n}
\end{equation}%
The cross product $\times $ is induced by a fixed orientation of $\mathbb{H}%
^{3}$ and $J^{\prime }$ denotes the covariant derivative of $J$ along $%
\gamma $ (the right hand side in the expressions are constant
functions, so they are well defined).

Let $\mathcal{M}$ be a submanifold of $\mathcal{L}$ and we take $[\gamma]\in \M$. 
Next we show that any tangent vector in $T_{[\gamma ]}\mathcal{M}$ corresponds 
(via $T_{\gamma }$) to a Jacobi vector field in $\mathcal{J}_{\gamma }$ associated with a
variation of $\gamma$ by unit speed geodesics whose equivalence classes are
in $\mathcal{M}$. In fact, given $X\in T_{[\gamma ]}\mathcal{M}$, there exists a smooth
curve $c:(-\varepsilon ,\varepsilon )\rightarrow \mathcal{M}$ with 
$c(0)=[\gamma ]$ and $\dot{c}(0)=X$. By Proposition 3 in \cite{Salvai-2007},
there exists a standard presentation of $c$, that is a function $\varphi :\mathbb{R}\times (-\varepsilon ,\varepsilon )
\rightarrow \mathbb{H}^{3}$ such that $s\mapsto \alpha _{t}(s):=\varphi (s,t)$ is a unit speed geodesic of 
$\mathbb{H}^{3}$ satisfying $c(t)=[\alpha _{t}]$, $\left\langle \dot{\beta}
(t),\dot{\alpha _{t}}(0)\right\rangle =0$ for all $t\in (-\varepsilon ,\varepsilon )$, where $\beta
(t)=\varphi (0,t)$, and $\varphi(0,0)=\gamma (0)$. It is easy to see that 
\begin{equation*}
J(s)=\left. {\frac{d}{dt}}\right\vert _{0}\alpha _{t}(s)
\end{equation*}%
is a Jacobi field in $\mathcal{J}_{\gamma }$ and it satisfies 
\begin{equation*}
T_{\gamma }(J)=\left. {\frac{d}{dt}}\right\vert _{0}[\alpha _{t}]=\left. {%
\frac{d}{dt}}\right\vert _{0}c(t)=X\text{.}
\end{equation*}%

\section{Global geodesic foliations of $\mathbb{H}^{3}$}

In this section we characterize, in terms of the canonical neutral metrics
on $\mathcal{L}$, the subsets $\mathcal{M}$ in $\mathcal{L}$ that determine
foliations of $\mathbb{H}^{3}$. To this end, it is convenient to give the
following definition.

\begin{definition}
A submanifold $\mathcal{M}$ of $\mathcal{L}$ is said to
be \emph{almost semidefinite} if $\Vert X\Vert _{\times }=0$ for $X\in T\mathcal{%
M}$ only if $\Vert X\Vert _{K}\geq 0$.
\end{definition}

\begin{theorem}
\label{adefinite} Let $\mathcal{M}$ be a surface contained in $\mathcal{L}$ 
\emph{(}the inclusion is a priori not even smooth\emph{)}. Then the
following statements are equivalent:

\begin{enumerate}
\item The surface $\mathcal{M}$ is the space of leaves of a smooth foliation
of $\mathbb{H}^{3}$ by oriented geodesics, with the canonical differentiable
structure.

\item The surface $\mathcal{M}$ is a closed almost semidefinite connected submanifold of $\mathcal{L}$.
\end{enumerate}
Besides, if $\M$ satisfies \emph{(a)} or \emph{(b)}, $\M$ is diffeomorphic to $\R^2$.
\end{theorem}

\medskip

Let $o$ be a fixed point in $\mathbb{H}^{3}$. We recall that 
\begin{equation*}
T(T_{o}^{1}\mathbb{H}^{3})=\{(u,v)\in T_{o}^{1}\mathbb{H}^{3}\times T_{o}%
\mathbb{H}^{3} : \left\langle u,v\right\rangle =0\}\cong TS^{2}\text{.}
\end{equation*}%
Let $f:\mathcal{L}\rightarrow \mathbb{H}^{3}$ be the map that assigns to
each oriented unit speed geodesic $\ell $ of the hyperbolic space its
closest point to $o$. Considering the following diagram, 
\begin{equation*}
\begin{array}{ccc}
T\left( T_{o}^{1}\mathbb{H}^{3}\right) & \overset{H}{\longrightarrow } & 
\mathcal{L} \\ 
\pi _{2}\downarrow \ \ \ \  &  & \ \ \downarrow f \\ 
T_{o}\mathbb{H}^{3} & \overset{\text{Exp}_{o}}{\longrightarrow } & \mathbb{H}%
^{3}%
\end{array}%
\end{equation*}%
we have that $f$ is a smooth map, where $H$ is the diffeomorphism given in (%
\ref{H}) and $\pi _{2}$ is the projection onto the second component.

Let $D:\mathcal{L}\rightarrow \mathbb{R}$ be the square distance from $o$.
In particular, if $\ell =H\left( u,v\right) $, we have that $D(\ell
)=|v|^{2} $ and so $D$ is smooth.

For any unit speed geodesic $\gamma$, let $\psi _{\gamma }:T_{[\gamma] }
\mathcal{L}\simeq \mathcal{J}_{\gamma }\rightarrow \dot{\gamma}(0)^{\bot }$
be the linear map defined by $\psi _{\gamma }(J)=J(0)$.

\begin{lemma}
\label{distancia} 
Let $\mathcal{M}$ be an almost semidefinite closed connected
two-dimensional submanifold of $\mathcal{L}$.

\begin{enumerate}
\item For any $\ell=[\gamma] \in \mathcal{M}$, $\left.{\psi_{\gamma}}%
\right\vert_{T_{\ell}\mathcal{M}}$ is surjective.

\item Any critical point $\ell $ of $\left. D\right\vert _{\mathcal{M}}$ is
a strict local minimum of $\left. D\right\vert _{\mathcal{M}}$ with $D(\ell
)=0$. Moreover, $D(\ell _{n})\rightarrow \infty $ as $n\rightarrow \infty $
for any sequence $\ell _{n}$ in $\mathcal{M}$ without cluster points.
\end{enumerate}
\end{lemma}

\begin{proof}
\noindent (a) It suffices to show that the map is injective 
($T_{\ell }\mathcal{M}$ and $\dot{\gamma}(0)^{\bot }$ have the same dimension). 
If $\psi _{\gamma}(J)=0$, then $J(0)=0$ and from (\ref{n}) we have that $\Vert T_{\gamma
}(J)\Vert _{\times}=0$. Since $\mathcal{M}$ is almost semidefinite, using (\ref{n}%
) we obtain that $J^{\prime }(0)=0$, thus $J\equiv 0$. 

\medskip

\noindent (b) Let $(u,v)\in T(T_{o}^{1}\mathbb{H}^{3})$ and let $H(u,v)=\ell$. 
In particular, $\ell =[\gamma _{U}]$ with $U=\left. \tau _{\gamma
_{v}}\right\vert _{0}^{1}(u)$, where $\tau _{\gamma _{v}}$ denotes the
parallel transport along $\gamma _{v}$. By $(\text{a})$, there exists a
Jacobi vector field $J\in \mathcal{J}_{\gamma _{U}}$, with $T_{\gamma
_{U}}\left( J\right) \in T_{\ell }\mathcal{M}$, such that $J(0)=\dot{\gamma}%
_{v}(1)$. We take a variation of $\gamma _{U}$ by unit speed geodesics $%
\Gamma (s,t)=\gamma _{t}(s)$ associated with $J$, with $[\gamma
_{t}]=H(u_{t},v_{t})\in \mathcal{M}$. We call $\alpha $ the smooth curve in $%
\mathcal{M}$ given by $\alpha (t)=[\gamma _{t}]$. We have $v_{t}=\pi
_{2}\circ \,H^{-1}\circ \,\alpha (t)$, thus $v_{t}$ is a smooth curve in $%
T_{o}\mathbb{H}^{3}$.

Suppose that $\ell \in \mathcal{M}$ is a critical point of $\left.
D\right\vert _{\mathcal{M}}$. First we verify that $D(\ell )=0$. We compute

\begin{equation}
0=\dot{\alpha}\left( 0\right) (D)=\left. \dfrac{d}{dt}\right\vert
_{0}D(\alpha (t))=\left. \frac{d}{dt}\right\vert _{0}|v_{t}|^{2}=2\langle
v,v_{0}^{\prime }\rangle \text{.}  \label{pi}
\end{equation}
Now, let $\lambda _{t}$ be the smooth curve in $\mathbb{R}$ such that $\text{%
Exp}_{o}(v_{t})=\gamma _{t}(\lambda _{t})$ (in particular, $\lambda _{0}=0$)
and consider the Jacobi vector field $K$ associated with the geodesic
variation $(s,t)\rightarrow \Delta (s,t)=\Gamma (s+\lambda _{t},t)$, that
is, 
\begin{equation*}
K(s)=\lambda _{0}^{\prime }\,\dot{\gamma}_{U}(s)+J(s)\text{.}
\end{equation*}%
Since $\text{Exp}_{o}(v_{t})=\Delta (0,t)$, we have 
\begin{equation}
(d\,\text{Exp}_{o})_{v}(v_{0}^{\prime })=K(0)=\lambda _{0}^{\prime }\,U+J(0)%
\text{.}  \label{v'0}
\end{equation}%
Besides, $(d\,\text{Exp}_{o})_{v}(v)=J(0)$. Then, by (\ref{pi}), the Gauss
Lemma and (\ref{v'0}), we obtain 
\begin{equation}
0=\langle v,v_{0}^{\prime }\rangle =\langle \left( d\,\text{Exp}_{o}\right)
_{v}(v),\left( d\,\text{Exp}_{o}\right) _{v}(v_{0}^{\prime })\rangle
=|J(0)|^{2}=|v|^{2}=D(\ell ),  \label{D}
\end{equation}%
as desired. Next we see that $\ell $ is a strict local minimum. Let $X$
be a nonzero vector in $T_{\ell }\mathcal{M}$ and let $J\in \mathcal{J}%
_{\gamma _{u}}$ such that $X=T_{\gamma _{u}}(J)$. Since $J$ is not an 
identically zero Jacobi vector field, by (a) we have $J(0)\neq 0$. As above,  
we take a smooth curve $[\gamma _{t}]=H(u_{t},v_{t})$ in 
$\mathcal{M}$ such that its initial velocity is $X$ and $J(s)=\left. \frac{
d}{dt}\right\vert _{0}\gamma _{t}(s)$. Then, 
\begin{equation*}
\left. \frac{d^{2}}{dt^{2}}\right\vert _{0}D([\gamma _{t}])=2(\langle
v_{0},v_{0}^{\prime \prime }\rangle +|v_{0}^{\prime }|^{2})=2|v_{0}^{\prime
}|^{2}>0,
\end{equation*}%
since $v_0 =v=0$ by (\ref{D}) and $v_{0}^{\prime }\neq 0$ by (\ref{v'0}).

The last statement is proved in a similar way as in Lemma 5(b) of \cite%
{Salvai-2009}.
\end{proof}

\medskip

\begin{proof}[Proof of Theorem \ref{adefinite}]
\noindent (a)\thinspace $\Rightarrow $\thinspace(b)
Suppose that the foliation is given by a smooth unit vector field $V$ on 
$\mathbb{H}^{3}$ and let $P:\mathbb{H}^{3}\rightarrow \mathcal{M}$ be the
smooth submersion induced by $V$ as in Proposition \ref{est.dif}. Since 
$\mathcal{M}=P(\mathbb{H}^{3})$, we have that $\mathcal{M}$ is connected. 
The fact that the inclusion $i:\mathcal{M}\rightarrow \mathcal{L}$ is a submanifold 
is proved in the same way as in the Euclidean case (the beginning of (a)\thinspace $\Rightarrow$\thinspace(b) 
in the proof of Theorem 2 in \cite{Salvai-2009}).

Let us see that $\mathcal{M}$ is almost semidefinite. Let $X\in T_{[\gamma ]}%
\mathcal{M}$ with $\Vert X\Vert _{\times }=0$ and let $J\in \mathcal{J}%
_{\gamma }\mathcal{\ }$with $X=T_{\gamma }(J)$. We want to see that $%
\left\Vert X\right\Vert _{K}\geq 0$. First, we observe that if $\gamma _{t}$
is any variation of $\gamma $ by geodesics in the foliation, associated with $%
J $, we have that $\dot{\gamma}_{t}\left( s\right) =V\left( \gamma
_{t}\left( s\right) \right) $ and so we compute 
\begin{equation}
J^{\prime }(s)=\frac{D}{ds}\left. \frac{d}{dt}\right\vert _{0}\gamma
_{t}\left( s\right) =\left. \frac{D}{dt}\right\vert _{0}\frac{d}{ds}\gamma
_{t}\left( s\right) =\left. \frac{D}{dt}\right\vert _{0}\dot{\gamma}%
_{t}\left( s\right) =\nabla _{J(s)}V\text{. }  \label{Jprima}
\end{equation}%
By (\ref{n}), $J\left( 0\right) $ and $J^{\prime }\left( 0\right) $ are
linearly dependent. If $J(0)=0$, then $J^{\prime }(0)=0$ by (\ref{Jprima}),
and so $\Vert X\Vert _{K}=0$. If $J(0)\neq 0$, there exists $a\in \mathbb{R}$
such that $J^{\prime }(0)=aJ(0)$. So, $J(s)=(a\, \sinh s + \cosh s)Z(s)$, where $Z$ is a 
parallel vector field along $\gamma$ and orthogonal to $\dot{\gamma}$. If $|a|>1$ 
there exists $s_o=(\tanh)^{-1} (-1/a)$ such that $J(s_o)=0$. By (\ref{Jprima}), we have that
$J'(s_o)=0$. Hence, $J\equiv 0$, which is a contradiction. Therefore, $|a|\leq 1$ and consequently
$\left\Vert X\right\Vert _{K}=(1-a^2)|J(0)|^2\geq 0$, as desired.

Next we show that $\mathcal{M}$ is closed. Let $[\gamma _{n}]=H(u_{n},v_{n})$
be a sequence in $\mathcal{M}$ with $\lim_{n\rightarrow \infty }[\gamma
_{n}]=[\gamma ]\in \mathcal{L}$. Let $(u,v)\in T(T_{o}^{1}\mathbb{H}^{3})$
such that $[\gamma ]=H(u,v)$. Since $H$ is a diffeomorphism we have that $%
(u_{n},v_{n})\rightarrow (u,v)$. Let $\bar{H}:T(T_{o}^{1}\mathbb{H}%
^{3})\rightarrow T^{1}\mathbb{H}^{3}$ be the smooth map defined by $\bar{H}%
(u,v)=\left. {\tau _{\gamma _{v}}}\right\vert _{0}^{1}(u)$ and recall that $%
H=\Pi \circ \bar{H}$ holds by definition of $H$, where $\Pi :T^{1}\mathbb{H}%
^{3}\rightarrow \mathcal{L}$ is the canonical projection. Since $[\gamma
_{n}]\in \mathcal{M}$, $\bar{H}(u_{n},v_{n})=V(\text{Exp}_{o}(v_{n}))$. So,
to prove that $\mathcal{M}$ is closed we have to see that $\bar{H}(u,v)=V(%
\text{Exp}_{o}(v))$. Now, the assertion follows from the continuity of $\bar{%
H}\text{, Exp}_{o}$ and $V$.

\bigskip

\noindent (b)\thinspace $\Rightarrow $\thinspace(a) The facts that the
union of all geodesics in $\mathcal{M}$ covers the whole space $\Hy$ and that two
distinct geodesics in $\mathcal{M}$ do not intersect are proved in a similar
way as in Theorem 2 of \cite{Salvai-2009}, but using in this case Lemma \ref%
{distancia} (b). As in that theorem, the hypotheses force the existence of
only one critical point (cf.\ the second paragraph of Remark \ref{closed}) and that $\M$ is 
diffeomorphic to $\R^2$.

Next, we define the vector field $V$ which determines the foliation. Given $%
z\in \mathbb{H}^{3}$, let $V(z)=\dot{\gamma}(t)$, where $[\gamma ]$ is the
unique element in $\mathcal{M}$ such that $z$ is in the trajectory of $%
\gamma $ and $z=\gamma (t)$. Now, we verify that $V$ is smooth. The
arguments differ from those in the Euclidean case only at the end, but we
include the details for the sake of completeness. The image of $V$ coincides
with $\Pi ^{-1}(\mathcal{M})$, and hence it is a smooth submanifold of $T^{1}%
\mathbb{H}^{3}$, since $\Pi $ is a fiber bundle. We have to check that zero
is the only vertical (with respect to $p:T^{1}\mathbb{H}^{3}\rightarrow 
\mathbb{H}^{3}$) tangent vector $\eta $ of the image of $V$. Suppose that $%
(dp)_{V(z)}(\eta )=0$ and let $t\mapsto V\circ c(t)$ be a smooth curve in $%
T^{1}\mathbb{H}^{3}$ such that $c(0)=z$ and with initial velocity equal to $%
\eta $. So, we have that $c^{\prime }(0)=0$. Let $\ell $ be the curve in $%
\mathcal{M}$ defined by $\ell (t)=\Pi (V(c(t)))$ and set $\ell ^{\prime
}(0)=X$. Let $\ell (0)=[\gamma ]$ with $\gamma (0)=c(0)$ and let $%
J(s)=\left. {\frac{d}{dt}}\right\vert _{0}\gamma _{V(c(t))}(s)$. We compute $%
J(0)=c^{\prime }(0)=0$ and we have that $J^{\prime }\left( 0\right) $ is orthogonal
to $\dot{\gamma}\left( 0\right) $, since $V$ is a unit vector field. Hence, $%
X=T_{\gamma }(J)$ and $\Vert X\Vert _{\times }=0$ by (\ref{n}) and so $%
\left\Vert X\right\Vert _{K}\geq 0$ since $\mathcal{M}$ is almost semidefinite.
This implies, again by (\ref{n}), that $J^{\prime }(0)=0$. Finally, if we
consider the isomorphism 
\begin{equation}
(dp_{V(z)},\mathcal{K}_{V(z)}):T_{V(z)}T\mathbb{H}^{3}\rightarrow T_{z}%
\mathbb{H}^{3}\times T_{z}\mathbb{H}^{3},  \label{isophi_v}
\end{equation}%
where $\mathcal{K}_{V(z)}$ is the connection operator, we obtain that $\eta $
is equal to zero, since $(dp_{V(z)},\mathcal{K}_{V(z)})(\eta
)=(J(0),J^{\prime }(0))=(0,0)$.
\end{proof}

\medskip

\begin{remark}\label{closed}
We construct in Proposition \ref{ejemplo} below an example of a two 
dimensional submanifold $\mathcal{M}$ of $\mathcal{L}$ satisfying 
all conditions of part (b) in Theorem \ref{adefinite}, except
to be closed. The geodesics in $\mathcal{M}$ not only fail to foliate the
whole $\mathbb{H}^{3}$, as expected, but they do not even foliate the open
set $\mathcal{U}$ in $\mathbb{H}^{3}$ given by the union of all their
trajectories (there exist two geodesics in $\mathcal{M}$ intersecting at a
point in $\mathcal{U}$).

The same proposition shows that if the hypothesis that $\mathcal{M}$ is
closed in $\mathcal{L}$ is removed in Lemma \ref{distancia}, there might
exist two different critical points $\ell _{1}$ and $\ell _{2}$ of $\left.
D\right\vert _{\mathcal{M}}$. One can take $o=f\left( 2,0\right) $ and as $%
\ell _{1}$ and $\ell _{2}$ the geodesics through $o$ with initial velocities 
$V_{\lambda}\left( 2,0\right) $ and $V_{\lambda}\left( 2,2\pi \right) $.
\end{remark}

\medskip 

We begin by defining an immersion $f$ of an open set of the plane into 
a totally geodesic submanifold $S$ of $\mathbb{H}^{3}$ covering an annulus in $S$ in a
non-injective way. Let $U$ be an open set in the half plane $\left\{ \left(
r,t\right) \mid r>0\right\} $ containing the rectangle $R=\left[ 1,3\right]
\times \left[ -\delta ,2\pi +\delta \right] $. Fix $o\in \mathbb{H}^{3}$ and
define $f:U\rightarrow \mathbb{H}^{3}$ by 
\begin{equation*}
f\left( r,t\right) =\text{Exp}_{o}\left( r\cos t~u_{o}+r\sin t~v_{o}\right) 
\text{,}
\end{equation*}%
where $u_{o},v_{o}\in T_{o}\mathbb{H}^{3}$ are unit orthogonal vectors. We
consider vector fields $u,v,w$ along $f$ forming an orthonormal basis of $%
T_{f\left( r,t\right) }\mathbb{H}^{3}$ for each $r,t$. They are given by 
\begin{equation*}
u=\frac{\partial f}{\partial r},\ \ \ \ v=\frac{1}{\sinh r}\frac{\partial f}{%
\partial t},\ \ \ \ w=u\times v\text{.}
\end{equation*}%
Now, let $\alpha_{\lambda} \left( r,t\right) =\alpha _{0}+\lambda t-\lambda r$ for
some $\lambda >0$ and $\alpha _{0}\in (0,\pi /2)$, and let $V_{\lambda}$ be the vector
field along $f$ defined by $V_{\lambda}=\cos \alpha_{\lambda} ~v+\sin \alpha_{\lambda} ~w$. Let $R^{o}$ be
the interior of $R$.

\begin{proposition}\label{ejemplo} 
For some $\lambda>0$, the map 
\begin{equation*}
F_{\lambda}:R^{o}\rightarrow \mathcal{L}\text{, \ \ \ \ \ }F_{\lambda}\left( r,t\right) =\left[
\gamma _{V_{\lambda}\left( r,t\right) }\right]
\end{equation*}%
is an immersion and $g_{\times }$ induces a Riemannian metric on its image $%
F_{\lambda}\left( R^{o}\right) =\mathcal{M}$. Moreover, the trajectories of $\gamma
_{V_{\lambda}(r,0)}$ and $\gamma _{V_{\lambda}(r,2\pi )}$ intersect at $f(r,0)$, for each $r\in
(1,3)$.
\end{proposition}

\begin{proof}
We fix $\left( r,t\right) \in R^{o}$ and $0\neq \left( x,y\right) \in T_{\left( r,t\right) }R^{\circ}$. 
For the sake of simplicity we omit the subindex $\lambda$ and write $\alpha$ instead of $\alpha_{\lambda}(r,t)$. 
Let us see that $\left\Vert dF_{\left( r,t\right) }\left( x,y\right) \right\Vert _{\times }>0$. Let $J$
be the Jacobi field associated with the variation $s\mapsto \gamma _{V\left(
r+sx,t+sy\right) }$. We compute that $J\left( 0\right) =x\,u+y\sinh r\,v.$
Now, since%
\begin{equation*}
\nabla _{u}u=0,\,\nabla _{u}v=0,\,\nabla _{u}w=0,\,\nabla _{v}u=\left( \coth
r\right) v,\,\nabla _{v}v=-\left( \coth r\right) u\,\,\text{and}\,\,\nabla
_{v}w=0\text{,}
\end{equation*}%
we obtain that 
\begin{equation*}
J^{\prime }(0)=-\left( y\cosh r\cos \alpha \right) u+\lambda (x-y)\left(
\sin \alpha \right) v-\lambda (x-y)\left( \cos \alpha \right) w\text{.}
\end{equation*}%
On the other hand, calling $\gamma =\gamma _{V(r,t)}$, we have that
\begin{equation*}
\dot{\gamma}(0)\times J(0)=V(r,t)\times J(0)=-\left( y\sinh r\sin \alpha
\right) u+\left( x\sin \alpha \right) v-\left( x\cos \alpha \right) w\text{.}
\end{equation*}%
Then, since the expression for $g_{\times }$ in (\ref{n}) is valid also if $J
$ is not orthogonal to $\dot{\gamma}$, we have that 
\begin{equation*}
\left\Vert dF_{\left( r,t\right) }\left( x,y\right) \right\Vert _{\times
}=\langle \dot{\gamma}(0)\times J(0),J^{\prime }(0)\rangle =\lambda
x^{2}-\lambda xy+\tfrac{1}{4}\left( \sinh 2r\sin 2\alpha \right) \,y^{2}.
\end{equation*}%
Thus, for $\mathcal{M}$ to be Riemannian, it is enough that $\lambda$
makes this bilinear form positive definite for all $\left( r,t\right)
\in R$. Equivalently, that $h_{\lambda }\left( r,t\right) >0$ for all $%
\left( r,t\right) \in R$, where for each $\lambda >0$, $h_{\lambda
}:R\rightarrow \mathbb{R}$ is defined by 
\begin{equation*}
h_{\lambda }\left( r,t\right) =\sinh \left( 2r\right) \sin \left( 2\left(
\alpha _{0}+\lambda t-\lambda r\right) \right) -\lambda \text{.}
\end{equation*}%
Now, $h_{\lambda }$ converges pointwise (and also uniformly, since $R$ is
compact) to $\sinh \left( 2r\right) \sin \left( 2\alpha _{0}\right) $ as $%
\lambda \rightarrow 0^{+}$. Since the limit function is positive, for $%
\lambda >0$ small enough, $h_{\lambda }\left( r,t\right) >0$ for all $\left(
r,t\right) \in R$, as desired.
\end{proof}

\section{Global nondegenerate geodesic foliations of $\mathbb{H}^{3}$}

Two unit speed geodesics $\gamma $ and $\alpha $ of $\mathbb{H}^{3}$ are
said to be \emph{asymptotic} if there exists a positive constant $C$ such
that $d(\gamma (s),\sigma (s))\leq C$, $\forall s\geq 0$ \cite{Eberlein-1996}. 
Two unit vectors $v,w\in T^{1}\mathbb{H}^{3}$ are said to be asymptotic if
the corresponding geodesics $\gamma _{v}$ and $\gamma _{w}$ have this
property.

A point at infinity for $\mathbb{H}^{3}$ is an equivalence class of
asymptotic geodesics of $\mathbb{H}^{3}$. The set of all points at infinity
for $\mathbb{H}^{3}$ is denoted by $\mathbb{H}^{3}(\infty )$ and has a
canonical differentiable structure diffeomorphic to the $2$-sphere. The
equivalence class represented by a geodesic $\gamma $ is denoted by $\gamma
(\infty )$ and the equivalence class represented by the oppositely oriented
geodesic $s\mapsto \gamma (-s)$ is denoted by $\gamma (-\infty )$. Let $%
\varphi ^{\pm }:\mathcal{L}\rightarrow \mathbb{H}^{3}(\infty )$ be the
forward Gauss map (for $+$) and the backward Gauss map (for $-$), defined by $%
\varphi ^{\pm }([\gamma ])=\gamma (\pm \infty )$, which are smooth.

In the introduction we commented on some distinguished types of geodesic
foliations of $\mathbb{H}^{3}$, regarding to which extent they are in some
sense trivial in some directions. That motivates the following precise definitions. 
Before we recall that by Theorem \ref{adefinite}, any smooth geodesic foliation 
of $\Hy$ has an associated submanifold $\M$ of $\LH$.

\begin{definition} \label{nondeg}
We say that a smooth foliation by oriented
geodesics of $\mathbb{H}^{3}$ is \emph{semi-nondegenerate} if the Gauss maps 
$\varphi ^{\pm }:\mathcal{M}\rightarrow \mathbb{H}^{3}(\infty )$ are local
diffeomorphisms, where $\mathcal{M}\subset \mathcal{L}$ is the space of
leaves. And we say that it is \emph{nondegenerate} if the only eigenvectors of $\nabla V$ are
in $\R V$, where $V$ is the unit vector field that determines the foliation.
\end{definition}

In order to characterize the semi-nondegenerate and nondegenerate global geodesic foliations of the
hyperbolic space in terms of the geometry of $\mathcal{L}$, we have the next
definition.

\begin{definition} 
A submanifold $\mathcal{M}$ of $\mathcal{L}$
is said to be \emph{semidefinite} if $\Vert X\Vert _{\times }=0$
for a nonzero $X\in T\mathcal{M}$ only if $\Vert X\Vert _{K}>0$.
\end{definition} 

\begin{theorem}\label{SAD} 
Let $\M$ be the submanifold of $\LH$ associated with a foliation of $\Hy$ by oriented geodesics.
Then 
\begin{enumerate}
\item the foliation is semi-nondegenerate if and only if $\M$ is semidefinite.
\item the foliation is nondegenerate if and only if $g_\times $ induces on $\M$ a definite metric.
\end{enumerate}
\end{theorem}

\medskip

Some definitions and lemmas will be necessary to prove the theorem.

A Jacobi vector field $J$ along a geodesic $\gamma $ of $\mathbb{H}^{3}$ is
said to be \emph{stable} (\emph{unstable}) if there exists a constant $c>0$
such that 
\begin{equation*}
|J(s)|\leq c,\hspace{0.5cm}\forall s\geq 0\hspace{0.5cm}(\forall s\leq 0).
\end{equation*}%
It is well-known that a Jacobi vector field $J$ along a geodesic 
$\gamma$ of $\mathbb{H}^{3}$ and orthogonal to $\dot{\gamma}$ is stable
(respectively, unstable) if and only if $J(s)=e^{-s}U(s)$ (respectively, 
$J(s)=e^{s}U(s)$) for some parallel vector field $U$ along $\gamma $
orthogonal to $\dot{\gamma}$.

We recall that given $v\in T^{1}\mathbb{H}^{3}$ and any point $p\in \mathbb{H%
}^{3}$, there exists a unique unit tangent vector at $p$ that is asymptotic
to $v$ (see \cite[Proposition 1.7.3]{Eberlein-1996}). A smooth vector field 
$W$ in $\mathbb{H}^{3}$ is called an asymptotic vector field if $W(p)$ and $W(q)$
are asymptotic for every $p,q\in \mathbb{H}^{3}$.

\begin{lemma}
Let $c$ be a smooth curve in $\mathbb{H}^{3}$. Then an asymptotic vector
field $W$ on $\mathbb{H}^{3}$ satisfies the following differential equation 
\begin{equation}
\nabla _{\dot{c}(t)}W=\langle \dot{c}(t),W_{c(t)}\rangle \,W_{c(t)}-\dot{c}%
(t)\text{.}  \label{Ec.Asintotica}
\end{equation}
\end{lemma}

\begin{proof} 
After decomposing $\dot{c}(t)$ into its components tangent and orthogonal to $W_{c(t)}$,
the statement follows directly from the following equations:
\begin{equation*}
\nabla _{X}\,W=-X,\,\text{if}\,\,X\bot \,W\hspace{0.5cm}\text{and}
\hspace{0.5cm}\nabla _{W}\,W=0.  \label{campoasint.}
\end{equation*}
The first one is true (see (1.10.9) in \cite{Eberlein-1996}), since it is well known
that the shape operator of a horosphere is the identity. The second
one holds, since the integral curves of an asymptotic vector field are
geodesics.
\end{proof}

\medskip

\begin{lemma}\label{jacobi} 
Let $\gamma $ be a geodesic of $\mathbb{H}^{3}$ and let $J\in 
\mathcal{J}_{\gamma }$ be given by $J(s)=\left. {\frac{d}{dt}}\right\vert
_{0}\gamma _{u_{t}}(s)$, where $t\mapsto u_{t}$ is a smooth curve in $T^{1}\mathbb{H}^{3}$, 
with foot points $c(t)$ \emph{(}in particular, $u_{0}=\dot{\gamma}\left( 0\right)
\bot \dot{c}\left( 0\right) $\emph{)}. If $v_{t}\in T_{\gamma (0)}^{1}%
\mathbb{H}^{3}$ is the asymptotic vector to $u_{t}$ for each $t\in \mathbb{R}
$, then the Jacobi vector field $K$ along $\gamma$ associated with $v_{t}$ satisfies 
\begin{equation*}
K^{\prime }(0)=J(0)+J^{\prime }(0).
\end{equation*}
\end{lemma}

\begin{proof}
For each $s,t\in \mathbb{R}$, let $V(s,t)$ be the unique unit vector at $%
c(s)$ that is asymptotic to $u_{t}$. In particular, $V(0,t)=v_{t}$ and $%
V(t,t)=u_{t}$. We compute 
\begin{equation}
\left. {\frac{D}{dt}}\right\vert _{0}u_{t}=\left. {\frac{D}{dt}}\right\vert
_{0}V(t,t)=\left. {\frac{D}{dt}}\right\vert _{0}V(t,0)+\left. {\frac{D}{dt}}%
\right\vert _{0}V(0,t)=\left. {\frac{D}{dt}}\right\vert _{0}V(t,0)+\left. {%
\frac{D}{dt}}\right\vert _{0}v_{t}  \label{der.parc}
\end{equation}%
The second equality follows from the well-known corresponding identity in
the calculus of several variables (writing $V$ in coordinates). The vector
field $V(t,0)$ is an asymptotic vector field along $t\mapsto c(t)$. So,
using that $J$ is orthogonal to $\dot{\gamma}$ (in particular $\left\langle 
\dot{c}\left( 0\right) ,V\left( 0,0\right) \right\rangle =0$) and (\ref%
{Ec.Asintotica}) with $W_{ c\left( t\right) } =V\left( t,0\right) $, we have
that 
\begin{equation}  \label{V(t,0)}
\left. {\frac{D}{dt}}\right\vert _{0}V(t,0)=-\dot{c}(0).
\end{equation}
Finally, since $J^{\prime }(0)=\left. {\frac{D}{dt}}\right\vert _{0}u_{t}$
and $K^{\prime }(0)=\left. {\frac{D}{dt}}\right\vert _{0}v_{t}$, by (\ref%
{der.parc}) and (\ref{V(t,0)}) we obtain $K^{\prime }(0)=J(0)+J^{\prime }(0)$, 
as desired.
\end{proof}

\begin{proof}[Proof of Theorem \ref{SAD}.]
\noindent (a) Suppose that the foliation is semi-nondegenerate. 
Let $[\gamma ]\in \mathcal{M}$ and let $0\neq X\in T_{[\gamma ]}\mathcal{M}$ such
that $\Vert X\Vert _{\times }=0$. We want to see that $\left\Vert
X\right\Vert _{K}>0$. By (a) $\Rightarrow $ (b) in Theorem \ref{adefinite}
we have $\Vert X\Vert _{K}\geq 0$. Suppose now that $\Vert X\Vert _{K}=0$.
Let $J\in \mathcal{J}_{\gamma }$ be the Jacobi vector field associated with 
$X$ via the isomorphism $T_{\gamma }$ given in (\ref{isoT}). Hence, by (\ref
{n}), $\{J(0),J^{\prime }(0)\}$ is linearly dependent and 
$\left\vert J^{\prime }\left( 0\right) \right\vert ^{2}=|J(0)|^{2}$. Since 
$X\neq 0$, we have that $J$ is a stable or an unstable vector field. If $J$
is a stable Jacobi vector field, by Proposition 1.10.7 in \cite{Eberlein-1996}, $J(s)=\left. 
\frac{d}{dt}\right\vert _{0}\gamma _{u_{t}}(s)$ for some smooth curve 
$t\mapsto u_{t}\in T^{1}\mathbb{H}^{3}$ with $u_{0}=\dot{\gamma}(0)$ and 
$u_{t}$ asymptotic for all $t$. Then, 
\begin{equation*}
(d\varphi ^{+})_{[\gamma ]}X=\left. \frac{d}{dt}\right\vert _{0}\varphi
^{+}([\gamma _{u_{t}}])=0,
\end{equation*}
which is a contradiction since $\varphi ^{+}$ is a local diffeomorphism. The case $J$ 
unstable is similar. Therefore, $\mathcal{M}$ is semidefinite.

Conversely, if $\M$ is semidefinite, we have to prove that the foliation is semi-nondegenerate, 
that is, that the backward and forward Gauss maps $\varphi ^{\pm }$ are local diffeomorphisms. 
So, let $[\gamma ]\in \mathcal{M}$ and $0\neq X\in T_{[\gamma ]}\mathcal{M}$. Let $%
J\in \mathcal{J}_{\gamma }$ be the Jacobi vector field associated with $X$
via (\ref{isoT}) and consider a smooth curve $t\mapsto u_{t}\in T^{1}\mathbb{%
H}^{3}$ such that $[\gamma _{u_{t}}]\in \mathcal{M}$ and $J(s)=\left. {\frac{%
d}{dt}}\right\vert _{0}\gamma _{u_{t}}(s)$. Let us prove that $(d\varphi
^{+})_{[\gamma ]}X\neq 0$ (the proof of the corresponding assertion for $%
\varphi ^{-}$ instead of $\varphi $ $^{+}$ is similar). We have to see that
the initial velocity of $t\mapsto \gamma _{u_{t}}(\infty )$ is different
from zero. By the definition of the differentiable structure of $\mathbb{H}%
^{3}(\infty )$ we have that the map that assigns to each $v\in T_{\gamma
(0)}^{1}\mathbb{H}^{3}$ the equivalence class of $\gamma _{v}$ in $\mathbb{H}%
^{3}(\infty )$ is a diffeomorphism. So, we consider the smooth curve $%
t\mapsto v_{t}\in T_{\gamma (0)}^{1}\mathbb{H}^{3}$ such that $\gamma
_{u_{t}}(\infty )=\gamma _{v_{t}}(\infty )$ and we show that $\left. \frac{d%
}{dt}\right\vert _{0}v_{t}\neq 0$. Let $K$ be the Jacobi vector field along $%
\gamma $ given by $K(s)=\left. {\frac{d}{dt}}\right\vert _{0}\gamma
_{v_{t}}(s)$ with initial conditions $K(0)=0$ and $K^{\prime }(0)=\left. {%
\frac{D}{dt}}\right\vert _{0}v_{t}$. By the isomorphism given in (\ref%
{isophi_v}), it suffices to see that $K^{\prime }(0)\neq 0$. Now, by Lemma %
\ref{jacobi}, $K^{\prime }(0)=J(0)+J^{\prime }(0)$. If $\Vert X\Vert
_{\times }\neq 0$, the set $\{J(0),J^{\prime }(0)\}$ is linearly independent
and so $K^{\prime }(0)\neq 0$. Now suppose that $\Vert X\Vert _{\times }=0$.
Since $X\neq 0$ and $\mathcal{M}$ is semidefinite, then 
$\left\Vert X\right\Vert _{K}=|J(0)|^2-|J'(0)|^2 > 0$. By (\ref{n}), the set $\{J(0),J^{\prime
}(0)\}$ is linearly dependent, and $J(0)\neq 0$. Hence, $J^{\prime
}(0)=\lambda J(0)$ with $\lambda \in \mathbb{R}-\{\pm 1\}$. Consequently, $%
K^{\prime }(0)=(1+\lambda )J(0)\neq 0$, as desired.

\medskip

\noindent(b) First, suppose that the foliation is nondegenerate, that is, the 
only eigenvectors of $\nabla V$ are in $\R V$. We want to see that $\|X\|_{\times}\neq 0$ 
for all $0\neq X\in T\M$. Suppose that there exists a nonzero vector $X\in T_{[\gamma]}\M$ such that
$\|X\|_{\times}=0$. Let $J\in \mathcal{J}_{\gamma}$ the Jacobi vector field associated with $X$ via the 
isomorphism $T_{\gamma}$ defined in (\ref{isoT}). By (\ref{Jprima}), $\nabla_{J(0)}V=J'(0)$ and since $X\neq 0$ we obtain that
$J(0)\neq 0$. Now, since 
$$
0=\|X\|_{\times}=\|T_{\gamma}(J)\|_{\times}=\left\langle \dot{\gamma}(0)\times
J(0), J'(0)\right\rangle, 
$$
we have that $J(0)\times J'(0)$ is orthogonal to $\dot{\gamma}(0)=V(\gamma(0))$. Or equivalently, 
$J'(0)$ is a multiple of $J(0)$. Again by (\ref{Jprima}), $J(0)$ is an eigenvector of 
$\nabla V$ orthogonal to $V(\gamma(0))$ (recall that $J\in \mathcal{J}_{\gamma}$), which is a contradiction.

Conversely, let $u\in T_{p}\Hy$ be an eigenvector of $\nabla V$ with eigenvalue $\lambda$. 
Let $c:(-\varepsilon, \varepsilon)\rightarrow \Hy$ be a smooth curve such that $\dot{c}(0)=u-\left\langle u, V(p)\right\rangle V(p) \bot V(p)$. 
So, the Jacobi vector field given by $J(s)=\left. \dfrac{d}{dt}\right\vert _{0}\gamma_{V(c(t))}\left( s\right)$ is in 
$\mathcal{J}_{\gamma}$, where $\gamma=\gamma_{V(p)}$. Since $J(0)=\dot{c}(0)$ and $\nabla _{V(p)}V=0$, 
we have that $\nabla_{J(0)} V=\lambda u$. As $\nabla_{J(0)} V=J'(0)$ by (\ref{Jprima}), 
$T_{\gamma}(J)\in T\M$ satisfies 
\begin{equation*}\label{nulo}
\|T_{\gamma}(J)\|_{\times}=\left\langle \dot{\gamma}(0)\times J(0), J'(0)\right\rangle = 
\left\langle V(p)\times \dot{c}(0), \lambda u\right\rangle=\left\langle V(p)\times u, \lambda u\right\rangle=0.
\end{equation*}
Since $g_{\times}$ induces on $\M$ a definite metric, we have that $T_{\gamma}(J)=0$ and so $\dot{c}(0)=J(0)=0$. Thus, 
$u$ is a multiple of $V(p)$, as desired.
\end{proof}

\medskip

The definition of semi-nondegenerate foliation says that geodesic varying smoothly within the foliation 
do not meet at infinity. The following theorem states that this local condition implies in fact the
global property that geodesics in the foliation do not meet at infinity at
all. In the proof we have to use coordinates in $\mathbb{H}^{3}$.
\medskip

\begin{theorem} 
Let $\mathcal{M}$ be the space of leaves of a
semi-nondegenerate smooth foliation of $\mathbb{H}^{3}$ by oriented geodesics.
Then the forward and backward Gauss maps $\varphi ^{\pm }:\mathcal{M}%
\rightarrow \mathbb{H}^{3}\left( \infty \right) $ are one to one. In
particular, they are diffeomorphisms onto their images.
\end{theorem}

\begin{proof} 
Let $P:\mathbb{H}^{3}\rightarrow \mathcal{M}$ be the map
assigning to each point $q$ in the hyperbolic space the oriented geodesic in
the foliation containing $q$, that is, $P\left( q\right) =\left[ \gamma
_{V\left( q\right) }\right] $. This is a fiber bundle with typical fiber $%
\mathbb{R}$. Since $\mathcal{M}$ is diffeomorphic to $\mathbb{R}^{2}$ by
Theorem \ref{adefinite}, there exists a global section $S:\mathcal{M}%
\rightarrow \mathbb{H}^{3}$. Let $F:\mathcal{M}^{\prime }\times \mathbb{R}%
\rightarrow \mathbb{H}^{3}$ be the diffeomorphism given by $F\left(
q,t\right) =\gamma _{V\left( q\right) }\left( t\right) $, where $\mathcal{M}%
^{\prime }=S\left( \mathcal{M}\right) \subset \mathbb{H}^{3}$. Let 
\begin{equation*}
F^{\pm }:\mathcal{M}^{\prime }\rightarrow \mathbb{H}^{3}\left( \infty
\right) \text{,\ \ \ \ \ \ \ \ \ \ \ }F^{\pm }\left( q\right) =\gamma
_{V\left( q\right) }\left( \pm \infty \right) \text{,}
\end{equation*}
which satisfies $F^{\pm }\circ S=\varphi ^{\pm }$. Clearly, it suffices to
prove that $F^{\pm }$ is one to one.

We consider the upper half space model of the hyperbolic space, that is,\\ 
$\left\{ \left( x,y,z\right) \mid z>0\right\}$ with the Riemannian metric 
$ds^{2}=\frac{1}{z^{2}}\left( dx^{2}+dy^{2}+dz^{2}\right) $. Without loss of
generality, we may suppose that $\left[ \gamma _{o}\right] \in \mathcal{M}$,
where $\gamma _{o}\left( t\right) =\left( 0,0,e^{t}\right) $, and that $%
S\left( \left[ \gamma _{o}\right] \right) =\left( 0,0,1\right) $. In this
model, $\mathbb{H}^{3}\left( \infty \right) =\left( \mathbb{R}^{2}\times
\left\{ 0\right\} \right) \cup \left\{ \infty \right\} $, $\varphi ^{+}\left[
\gamma _{o}\right] =\gamma _{o}\left( \infty \right) =\infty $ and $\varphi
^{-}\left[ \gamma _{o}\right] =\gamma _{o}\left( -\infty \right) =0$.

Since $\varphi ^{\pm }$ (or equivalently $F^{\pm }$) are local
diffeomorphisms, there exists a neighborhood $A$ of $\left( 0,0,1\right) $
in $\mathcal{M}^{\prime }$, and neighborhoods $U_{+}$ and $U_{-}$ of $\infty 
$ and $0$ in $\mathbb{H}^{3}\left( \infty \right) $, respectively, such that 
$F^{\pm }:A\rightarrow U_{\pm }$ is a diffeomorphism. Let $B_{+}\subset
U_{+} $ be the complement of a closed disk centered at $0$ of radius $R$ in $%
\mathbb{R}^{2}\times \left\{ 0\right\} $. Let $A^{\prime }\subset A$ and $%
B_{-}\subset U_{-}$ be such that $F^{\pm }:A^{\prime }\rightarrow B_{\pm }$
are diffeomorphisms. Taking, if necessary, a larger $R$, we may suppose that 
$B_{-}$ is contained in the disk of radius $\delta $ (also centered at $0$).

Let us see that $F\left( A^{\prime }\times \mathbb{R}\right) $ contains the
horoball $\left\{ \left( x,y,z\right) \mid 2z\geq R+\delta \right\} $. If 
$\partial A^{\prime }$ is the border of $A^{\prime }$ in $\mathcal{M}^{\prime
}$, then $F\left( \partial A^{\prime }\times \mathbb{R}\right) $ is a
cylinder in $\mathbb{H}^{3}$ separating the space in two connected
components, in such a way that $F\left( A^{\prime }\times \mathbb{R}\right)$
is the component containing the trajectory of $\gamma _{o}$. The assertion
follows from the fact that the cylinder is built up with trajectories of
geodesics in $\mathbb{H}^{3}$ (vertical semicircles with center in $\mathbb{R
}^{2}\times \left\{ 0\right\} $) whose $z$-component is smaller than $\frac{1
}{2}\left( R+\delta \right) $.

Finally, given $\left[ \sigma \right] \in \mathcal{M}$ such that $\sigma
\left( \infty \right) =\gamma _{o}\left( \infty \right) $, we want to see
that $\left[ \sigma \right] =\left[ \gamma _{o}\right] $. The geodesic 
$\sigma $ must have the form $\sigma \left( t\right) =\left(
x_{o},y_{o},z_{o}e^{t}\right) $ for some real numbers $x_{o},y_{o},z_{o}$,
with $z_{o}>0$. For $t$ large enough, $\sigma \left( t\right) $ is in the
horoball. In particular, there exists $t_{1}$ such that $\sigma \left(
t_{1}\right) =F\left( q,s\right) $ for some $\left( q,s\right) \in A^{\prime
}\times \mathbb{R}$. Hence $\sigma \left( t_{1}\right) =\gamma _{V\left(
q\right) }\left( s\right) $. Now, since for each point of $\mathbb{H}^{3}$
passes only one geodesic in $\mathcal{M}$, we have that $\left[ \sigma 
\right] =\left[ \gamma _{V\left( q\right) }\right] $. Consequently, $\gamma
_{V\left( q\right) }\left( \infty \right) =\infty $ and so $q=\left(
0,0,1\right) $, since $F^{+}$ is one to one on $A^{\prime }$. Therefore 
$\left[ \sigma \right] =\left[ \gamma _{o}\right] $. The injectivity of 
$\varphi ^{-}$ is verified in a similar way.
\end{proof}

\bibliographystyle{amsplain}
\bibliography{mybib}

\end{document}